\documentclass[final,onefignum,onetabnum]{siamonline250211}

\usepackage{lipsum}
\usepackage{amsfonts}
\usepackage{graphicx}
\usepackage{epstopdf}
\usepackage{subfig}
\usepackage{booktabs}

\usepackage[noend]{algpseudocode}

\ifpdf
  \DeclareGraphicsExtensions{.eps,.pdf,.png,.jpg}
\else
  \DeclareGraphicsExtensions{.eps}
\fi

\usepackage{enumitem}
\setlist[enumerate]{leftmargin=.5in}
\setlist[itemize]{leftmargin=.5in}

\newsiamremark{remark}{Remark}
\newsiamremark{hypothesis}{Hypothesis}
\crefname{hypothesis}{Hypothesis}{Hypotheses}
\newsiamthm{claim}{Claim}

\headers{Numerical Stability of the Nystr\"om Method}{A. Bucci, Y. Nakatsukasa, T. Park}

\title{Numerical Stability of the Nystr\"om Method\thanks{Submitted to the editors DATE.
\funding{AB is member of the INdAM Research Group GNCS and supported by the UK's Engineering and Physical Sciences Research Council (EPSRC grant EP/Z533786/1), YN is supported by EPSRC grants EP/Y010086/1 and EP/Y030990/1. TP is supported by the RandESC project, funded by the Swiss Platform for Advanced Scientific Computing (PASC).}}}

\author{Alberto Bucci\thanks{School of Mathematics, The University of Edinburgh, Edinburgh, EH9 3FD, UK
  (\email{abucci2@ed.ac.uk}).}
\and Yuji Nakatsukasa\thanks{Mathematical Institute, University of Oxford, Oxford, OX2 6GG, UK (\email{nakatsukasa@maths.ox.ac.uk}).}
\and Taejun Park\thanks{Institute of Mathematics, EPF Lausanne, 1015 Lausanne, Switzerland (\email{taejun.park@epfl.ch}). Part of the work was done while the author was at the University of Oxford.}}

\usepackage{amsopn}
\DeclareMathOperator{\diag}{diag}

\def\Pi{\prod}

\def\Ah{\widehat A}
\def\At{\widetilde A}

\def\Nystrom{Nystr{\"o}m }

\newcommand{\mbb}{\mathbb}
\newcommand{\mach}{u_*}

\renewcommand{\Re}{\mbb{R}}

\newcommand{\norm}[1]{\left\|{#1}\right\|}

\usepackage{stmaryrd}
\newcommand{\lowrank}[2]{\llbracket {#1} \rrbracket_{#2}}

\newcommand{\nyseps}{A_N^{(\epsilon)}}
\usepackage[normalem]{ulem}

\newcommand{\ignore}[1]{}

\ifpdf
\hypersetup{
  pdftitle={Numerical Stability of the Nystr\"om Method},
  pdfauthor={Alberto Bucci, Yuji Nakatsukasa, Taejun Park}
}
\fi

\begin{document}

\maketitle

\begin{abstract}
The Nystr\"om method is a widely used technique for improving the scalability of kernel-based algorithms, including kernel ridge regression, spectral clustering, and Gaussian processes. Despite its popularity, the numerical stability of the method has remained largely an unresolved problem. In particular, the pseudo-inversion of the submatrix involved in the Nystr\"om method may pose stability issues as the submatrix is likely to be ill-conditioned, resulting in numerically poor approximation. In this work, we establish conditions under which the Nystr\"om method is numerically stable. We show that stability can be achieved through an appropriate choice of column subsets and a careful implementation of the pseudoinverse. Our results and experiments provide theoretical justification and practical guidance for the stable application of the Nystr\"om method in large-scale kernel computations.
\end{abstract}



\section{Introduction} \label{sec:intro}
Low-rank techniques play an important role in improving scalability of many algorithms. Matrices that are low-rank are ubiquitous across data science and computational sciences \cite{UdellTownsend2019}. Among these techniques, the Nystr\"om method \cite{nystrom} stands out as a key tool in machine learning \cite{GM} for enhancing scalability in applications such as kernel ridge regression \cite{AlaouiMahoney2015}, kernel support vector machine \cite{CortesMohriTalwalkar2010}, spectral clustering \cite{WGM}, Gaussian processes \cite{SeegarWilliams,gpref},
and manifold learning \cite{TalwalkarKumarMohriRowley2013}. Beyond machine learning, the Nystr\"om approximation also finds use in statistics \cite{BelabbasWolfe2009}, signal processing \cite{ParkerWolfeTarokh2005}, and computational chemistry \cite{EpperlyTroppWebber2024}. 

Given a symmetric positive semi-definite (SPSD) matrix $A\in \Re^{n\times n}$, the Nystr\"om method \cite{SeegarWilliams,nystrom,GM} constructs a rank-$r$ approximation of $A$ as:
\begin{equation*}
    A \approx A_N:= AS(S^TAS)^\dagger S^T\!A,
\end{equation*} where $S \in \Re^{n\times r}$ is a column-sampling matrix
that selects a subset of $A$'s columns, that is, $S$ is a column submatrix of the identity matrix $I$, and $^\dagger$ denotes the Moore-Penrose pseudoinverse.
The approximation depends on the pseudoinverse of the core matrix $S^TAS$, which is almost always ill-conditioned for matrices well-approximated by a low-rank matrix. This may result in reduced accuracy due to roundoff errors, particularly if the implementation is not carefully designed. While these issues are not always critical, they can degrade the quality of the approximation to the point where we lose all accuracy, making the approximation pointless.

\begin{algorithm}
\caption{First implementation of Nystr\"om's method 
}
\label{alg:NaiveNystrom}
\begin{algorithmic}[1]
\Require{Symmetric positive semidefinite matrix $A\in \Re^{n\times n}$, a sketching matrix $S\in \Re^{n\times r}$}
\Ensure{
$B\in \Re^{n\times r}$ such that $A\approx BB^T$ }
\State Compute $C \gets AS$, $W \gets S^T\!C$,
\State Compute Cholesky factor $R^T R = W$,
\State Solve $B = C/R$
\end{algorithmic}
\end{algorithm}

To investigate potential numerical issues, we examine the implementation of the Nystr\"om method presented in \Cref{alg:NaiveNystrom}. In exact arithmetic, the algorithm produces $A_N$ exactly. However, in finite-precision arithmetic, roundoff errors may render the procedure numerically unstable. In particular, Step $2$ of \Cref{alg:NaiveNystrom}, which involves computing the pseudoinverse of $W$, is a potential source of instability when $W$ is highly ill-conditioned, which is typically the case when $A$ is well-approximated by a low-rank matrix. This motivates the need for an alternative algorithm and a numerically robust implementation of the Nystr\"om method.

To illustrate the issue, consider the matrix $A = \diag(1,\gamma^2,0)$, where $\gamma$ is on the order of machine precision. In exact arithmetic, selecting the first two columns of $A$ yields the optimal Nystr\"om approximation of $A$:
\begin{align*}
    \norm{A-A_N} &= \norm{\begin{bmatrix}
        1 & 0 & 0 \\ 0 & \gamma^2 & 0 \\ 0 & 0 & 0
    \end{bmatrix} - \begin{bmatrix}
        1 & 0 \\ 0 & \gamma^2 \\ 0 & 0
    \end{bmatrix} \begin{bmatrix}
        1 & 0 \\ 0 & \gamma^2 
    \end{bmatrix}^\dagger \begin{bmatrix}
        1 & 0 & 0 \\ 0  & \gamma^2 & 0
    \end{bmatrix}} = 0. 
\end{align*} However, in Step 2 if we get a slightly perturbed Cholesky factor due to roundoff errors, we get
\begin{equation*}
    \hat{R} = \begin{bmatrix}
        1 & 0 \\ 0 & \gamma + \tilde{\gamma}
    \end{bmatrix},
\end{equation*} where $\tilde{\gamma}$ is in the order of machine precision, instead of $R = \diag(1,\gamma)$ in exact arithmetic. Note that $\hat{R}^T\hat{R} = R^TR + \text{O}(\gamma^2) = S^T\!AS +\text{O}(\gamma^2)$, so $\hat{R}$ is a good substitute for $R$, in particular it is a backward stable Cholesky factor of $S^TAS$. With $\hat{R}$ we obtain
\begin{equation*}
    \hat{B} = AS\hat{R}^{-1} = \begin{bmatrix}
        1 & 0 \\ 0 & \gamma^2 \\ 0 & 0
    \end{bmatrix} \begin{bmatrix}
        1 & 0 \\ 0 & \gamma + \tilde{\gamma}
    \end{bmatrix}^{-1} = \begin{bmatrix}
        1 & 0 \\ 0 & \frac{\gamma^2}{\gamma+\tilde{\gamma}} \\ 0 & 0
    \end{bmatrix}.
\end{equation*}
Let the computed Nystr\"om approximation be $\widehat{A}_N:=\hat{B}\hat{B}^T$. Then we get
\begin{align*}
    \norm{A-\widehat{A}_N} = \norm{A-\hat{B}\hat{B}^T} &= \left|\gamma^2 - \frac{\gamma^4}{(\gamma+\tilde{\gamma})^2}\right|.
\end{align*}
Now set $\tilde{\gamma} = - (1-\delta)\gamma$ for $\delta \in (0,0.5)$ so that $\tilde{\gamma} = \text{O}(\gamma)$, then 
\begin{equation*}
    \norm{A-\widehat{A}_N} = \gamma^2\left(\frac{1}{\delta^2} - 1\right) = \text{O}\!\left(\frac{\gamma^2}{\delta^2}\right),
\end{equation*} making the Nystr\"om error arbitrarily large as $\delta \rightarrow 0$. Therefore, the Nystr\"om approximation may give an extremely poor approximation if the pseudoinverse of the core matrix is not computed with care.
Essentially, the issue arises from very small entries in the diagonal of the Cholesky factor, making the Nystr\"om approximation arbitrarily large when solving $C/R$ in line $3$ of \Cref{alg:NaiveNystrom}. This issue has also been raised in prior works; see \Cref{subsec:existing}.

In this work, we propose an alternative stabilization strategy based on $\epsilon$-truncation of the singular values of the core matrix, $S^T\!AS$. By discarding singular values below a prescribed tolerance $\epsilon$, this method ensures that the singular values of the core matrix remain sufficiently large before computing its pseudoinverse. We refer to this procedure as the stabilized Nystr\"om method, formally defined as follows: 
\begin{equation*}
    \nyseps := AS(S^T\!AS)_\epsilon^\dagger S^T\!A,
\end{equation*} where $(S^T\!AS)_\epsilon$ denotes the truncated version of $S^T\!AS$, obtained by setting all singular values smaller than $\epsilon$ to zero. This stabilized formulation is central to our subsequent analysis. In particular, we investigate its approximation accuracy and establish its numerical stability guarantees under a carefully designed implementation, as outlines in \Cref{Alg:stableNystrom}. 

\Cref{Alg:stableNystrom} presents our implementation of Nystr\"om method, which differs fundamentally from \Cref{alg:NaiveNystrom} in that it may produce a different output even in exact arithmetic. The main difference lies in Step $2$, where we replace the standard Cholesky factor with an $\epsilon$-truncated Cholesky factor. In this approach, the Cholesky process is terminated once the largest remaining diagonal element falls below the threshold $\epsilon$.
In our implementation, this procedure is performed in MATLAB using the $\texttt{cholp}$ function from the Matrix Computation Toolbox \cite{cholpcode}. Similar to \Cref{alg:NaiveNystrom}, \Cref{Alg:stableNystrom} provides the option to retain the large low-rank factor as a subset of columns of $A$ by skipping Step $3$ and outputting $C$ and $R_{\epsilon}$ as the low-rank factors.

\begin{algorithm}
\caption{Stable implementation of Nystr\"om's method}
\label{Alg:stableNystrom}
\begin{algorithmic}[1]
\Require{Symmetric positive semidefinite matrix $A\in \Re^{n\times n}$, a sketching matrix $S\in \Re^{n\times r}$, truncation parameter $\epsilon$ (recommendation: $\epsilon = 10u$)} 
\Ensure{
$B\in \Re^{n\times r}$ such that $A\approx BB^T$ }
\State Compute $C\gets AS$, $W \gets S^T\!AS$,
\State Compute $\epsilon$-truncated Cholesky factor $R_\epsilon^T R_\epsilon \gets W$ where $R_{\epsilon} \in \Re^{\hat{r} \times r}$ with $\hat{r} \leq r$,
\State Compute $B \gets CR_{\epsilon}^\dagger$ via backward stable solve for the overdetermined system
\end{algorithmic}
\end{algorithm}

\subsection{Existing methods} \label{subsec:existing}
Prior works have investigated the potential numerical instability of the Nystr\"om method when computing the pseudoinverse of the core matrix $S^T\!AS$.
In particular, Tropp et al. \cite{TroppYurtseverUdellCevher2017spsd} proposed introducing a small shift to the core matrix before computing its pseudoinverse; see Algorithm 3 in \cite{TroppYurtseverUdellCevher2017spsd}. The small shift ensures that the singular values of $S^T\!AS$ remain sufficiently larger than the unit roundoff, thereby improving numerical stability while preserving the spectral decay of $A$. 
Carson and Dau\u zickait\. e \cite{CarsonDauvzickaite2024} analyzed the numerical stability of this approach. While this strategy appears to be numerically stable in practice, their analysis is limited: the 
bound suggests that
roundoff error is amplified approximately by $\sigma_1(A)/\sigma_r(A)$, which can be extremely large for matrices with fast decaying singular values; this is often the case when the matrix is well-approximated by a low-rank matrix.
Moreover, the method requires computing the QR factorization of an $n\times r$ matrix, and the resulting low-rank factors cannot be kept as subsets of the original matrix, which can be considered a drawback. To our knowledge, shifting has been the only technique proposed to address the potential instability of the Nystr\"om method. We note that the shifted formulation does not yield the exact Nystr\"om approximation in exact arithmetic, as the introduced shift slightly perturbs the eigenvalues of the core matrix.

Related approaches have been used in the broader context of low-rank approximations. One common strategy is the use of the $\epsilon$-pseudoinverse, which truncates singular values below a threshold $\epsilon$ before (pseudo)inversion. For example, Chiu and Demanet \cite{curtrunc} applied this idea to the CUR decomposition to improve numerical stability, and subsequent analyses \cite{pn24cur} derived stability bounds for the CUR decomposition using an $\epsilon$-pseudoinverse. This idea has also been applied to generalized Nystr\"om method \cite{nakatsukasa2020fast}.

\subsection{Contributions} \label{subsec:contribution}
In this work, we study the numerical stability of the Nystr\"om method with \emph{locally maximum-volume indices} (see \Cref{def:maxvol}) by employing the $\epsilon$-pseudoinverse. We propose a numerically stable approach for computing the Nystr\"om approximation that truncates small singular values of the core matrix before taking its pseudoinverse as outlined in \Cref{Alg:stableNystrom}. This strategy avoids the numerical instabilities associated with standard pseudoinversion and eliminates the need for additional shifting strategies, while maintaining the accuracy of the approximation.

We provide stability analysis showing that the roundoff error of our approach is independent of the condition number of the original matrix $A$, ensuring high accuracy without stability concerns. 
This means one can safely use low-precision arithmetic 
to improve speed (e.g. single precision instead of the standard double precision), 
if a low accuracy is sufficient for the application, which is often the case in data science. This was suggested in~\cite{CarsonDauvzickaite2024}\footnote{
Carson and Dau\v{z}ickaitė~\cite{CarsonDauvzickaite2024} consider the use of mixed precision in implementing Nystr\"om's method, in particular using lower-precision in  computing the matrix-matrix multiplication $AS$. Here we do not explore this, partly for simplicity, and partly because in our context where $S$ is a subsampling matrix, $AS$ is simply a column submatrix of $A$, and hence can often be computed exactly.}, but our bounds crucially removes the dependence on the conditioning of $A$ or $\lowrank{A}{r}$, where $\lowrank{A}{r}$ is the best rank-$r$ approximation to $A$. 
While our analysis focuses on locally maximum-volume indices, we expect similar stability guarantees for any well-chosen set of representative indices such as \cite{RPChol}.

Furthermore, our method is computationally efficient and can preserve the original matrix structure such as sparsity or nonnegativity in the factors, by taking $S$ to be a subset selection matrix. 
In comparison to the shifting strategies, we avoid the QR factorization of an $n\times r$ matrix, allowing our algorithm to scale effectively to large matrices. Also, by relying on subsets of the original data, it retains interpretability and data fidelity. This combination of stability, efficiency and structural preservation makes our approach particularly suitable for large-scale computing.

\subsection{Notation}
Throughout, we use $\norm{\cdot}_2$ for the spectral norm or the vector-$\ell_2$ norm, $\norm{\cdot}_F$ for the Frobenius norm, and $\norm{\cdot}_*$ for the trace norm. We use dagger $^\dagger$ to denote the pseudoinverse of a matrix and $\lowrank{A}{r}$ to denote the best rank-$r$ approximation to $A$ in any unitarily invariant norm, i.e., the approximation derived from truncated SVD~\cite[\S~7.4.9]{hornjohnson}. Unless otherwise specified, $\sigma_i(A)$ denotes the $i$th largest singular value of the matrix $A$. We use MATLAB-style notation for matrices and vectors. For example, for the $k$th to $(k+j)$th columns of a matrix $A$ we write $A(:,k:k+j)$. Lastly, we use $|J|$ to denote the cardinality of the index set $J$ and define $[n]:=\{1,2,...,n\}$.

\section{Overview of the Nystr\"om method and Kernel matrices} \label{subsec:nystromKernel}
The Nystr\"om method is widely used in data science to efficiently approximate kernel matrices \cite{drineasmahoney05,GM,WGM}, especially in scenarios where computing and storing the full kernel matrix is computationally prohibitive. By leveraging low-rank approximations, the Nystr\"om method improves the scalability of kernel-based machine learning algorithms in both memory and computational costs. This section provides an overview of the Nystr\"om method and kernel matrices.

Given $n$ $d$-dimensional data points $z_1,z_2,...,z_n \in \Re^d$ and a kernel function $\kappa: \Re^d \times \Re^d \to \Re$, the kernel matrix of $\kappa$ with respect to $z_1,z_2,...,z_n$ is defined by
\begin{equation*}
    A_{ij} = \kappa(z_i,z_j).
\end{equation*} Appropriate choices of $\kappa$ ensure that $A$ is SPSD and we interpret the kernel function as a measure of similarity between a pair of data points. 
In this case, the kernel function determines a feature map $\Phi: \Re^d \to \mathbb{H}$ such that $A_{ij} = \langle\Phi(z_i),\Phi(z_j)\rangle_{\mathbb{H}}$. Here, $\mathbb{H}$ is a Hilbert space and dim$(\mathbb{H})$ can be much larger than $n$ or even infinite. Therefore, kernel matrices implicitly compute inner products in a high-dimensional feature space without explicitly transforming the data. 
When data are mapped to a high-dimensional feature space, it is often easier to extract nonlinear structure in the data. Machine learning algorithms where this is exploited through kernel matrices include kernel SVM \cite{SVMBook}, kernel ridge regression \cite{kernelbook2002} and kernel k-means clustering \cite{WGM}. Popular kernels include polynomial kernels, Gaussian kernels, and Sigmoid kernels. See, for example, \cite{kernelbook2002, HofmannScholkopfSmola2008kernelML} for details. 

For a given dataset, constructing the kernel matrix involves evaluating the kernel entries, a process that becomes increasingly expensive as $n$ and $d$ become large. This issue is particularly problematic when the kernel function is computationally expensive to compute. To address this, choosing $S\in \Re^{n\times r}$ as a column subsampling matrix in the Nystr\"om method has proven to be an effective strategy for improving the scalability of kernel methods. When $S$ is a column subsampling matrix, $AS$ is an $r$-subset of the columns of $A$, and $S^T\!AS$ is an $r\times r$ principal submatrix of $A$.  Consequently, the Nystr\"om approximation requires only $nr$ kernel evaluations instead of the full $n^2$. Therefore, we focus on column subsampling matrices in this work.\footnote{Other popular choices for $S$ include random embeddings such as Gaussian embeddings and sparse maps; see \cite{MartinssonTropp2020,GM}.}

There are several ways of choosing the important columns for the Nystr\"om approximation. The methods used to choose columns significantly impact the accuracy of the Nystr\"om method. There are numerous practical algorithms in the literature for constructing a subsampling matrix, including uniform sampling \cite{SeegarWilliams,drineasmahoney05}, leverage score sampling \cite{DrineasMahoneyMuthukrishnan2008,GM,MahoneyDrineas2009}, determinantal point process (DPP) sampling \cite{DPPsampling}, volume-based methods \cite{Massei2022}, and pivot-based strategies such as greedy pivoting \cite{GreedyChol}, adaptive randomized pivoting \cite{CortinovisKressner2024}, and randomly-pivoted Cholesky \cite{RPChol}. Notably, certain methods produce columns that provide close-to-optimal (quasi optimal) guarantees \cite{Massei2022,CortinovisKressner2024}:
\begin{equation}\label{eq:nystacc}
    \norm{A-A_N}_* \leq (1+r) \norm{A-\lowrank{A}{r}}_*.
\end{equation} It is worth noting that the suboptimality in~\eqref{eq:nystacc} does not depend on $n$; 
 the existence of a Nystr\"om approximation satisfying~\eqref{eq:nystacc} for an arbitrary PSD matrix $A$ is arguably remarkable.

Although many effective strategies exist for identifying the important columns of an SPSD matrix, this work focuses on \emph{locally maximum-volume indices} (locally max-vol indices) \cite{Massei2022,GOSTZ18}. The formal definition is given below.


\begin{definition} \label{def:maxvol}
    Let $A\in \Re^{n\times n}$ be an SPSD matrix. A subset $J\subseteq [n]$ with $|J| = r$ is called a set of locally max-vol indices if
    \begin{equation*}
       \mathrm{det}(A(\hat{J},\hat{J})) \leq  \mathrm{det}(A(J,J))
    \end{equation*} for every $\hat{J}\subseteq [n]$ with $|J\cap \hat{J}| = r-1$, and $|\hat{J}|=r$.
\end{definition}
Since $A$ is SPSD, $\mathrm{det}(A(J,J))$ equals the volume of the submatrix $A(J,J)$. A locally max-vol index set is therefore locally optimal in the sense that its volume cannot be increased by replacing a single index.

Locally max-vol indices enjoy several useful approximation guarantees. In particular, they satisfy the following quasi-optimal bound in the max-norm \cite[Lemma 1]{Massei2022}:
\begin{equation*}
    \norm{A-A_N}_{\max} \leq (1+r)\sigma_{r+1}(A),
\end{equation*} where $\norm{\cdot}_{\max}$ denotes the maximum absolute entry of the argument.

Compared with globally maximum-volume (max-vol) indices, i.e., 
\begin{equation*}
    J_* \in\underset{J\in [n],|J|=r}{\mathrm{argmax}}\mathrm{det}(A(J,J)),
\end{equation*} locally max-vol indices offer important practical advantages. Constructing a globally max-vol set (or approximate max-vol set) is NP-hard \cite{civrilmagdon-ismail09} and therefore computational infeasible in general, whereas verifying and improving a set of indices to a locally max-vol index can be done in $\text{O}(nr^2\log r)$ time \cite{Massei2022}. Locally max-vol indices also yield high-quality low-rank approximations of SPSD matrices. Since locally max-vol indices are highly effective, the analysis in this work serves as a good proxy for understanding the behaviour of any high-quality set of indices.

In this paper we focus on a locally max-vol index set $J$ and its associated subsampling matrix $S$, defined so that $A(:,J) = AS$. We establish stability of the resulting Nystr\"om's method, when implemented as in \Cref{Alg:stableNystrom}. Although the (locally) max-vol choice has historically been the gold standard, leading to a number of desirable properties, it does not necessarily minimize the error $\|A-A_N\|_*$, and is not guaranteed to satisfy~\eqref{eq:nystacc}. Other choices of (still subsampling) $S$ have been introduced to satisfy such bounds \cite{CortinovisKressner2024,osinsky,DPPsampling}. Nevertheless, we believe that, just as with (locally) max-vol $S$, Nystr\"om method remains stable under any `good' choice of $S$ when implemented according to \Cref{Alg:stableNystrom}.

\section{Analysis of the stabilized Nystr\"om method in exact arithmetic}

Here we analyze the approximation accuracy (or error) of the \Nystrom method with $\epsilon$ pseudoinverse, which from now on we will call stabilized \Nystrom (SN), assuming that the subsampling matrix $S$ is a locally max-vol subsampling for $A$. The analysis is first carried out in exact arithmetic, neglecting round-off errors, while stability in finite precision will be addressed in the following section.

The structure of the analysis closely follows \cite{nakatsukasa2020fast}, where the generalized \Nystrom is analyzed. However, the techniques we employ will differ as matrices with different properties are involved, and unlike generalized \Nystrom where two \emph{independent} sketches of $A$ are employed, in Nystr\"om we use a single sketch (and hence is a case of generalized \Nystrom where the sketches are highly dependent).

\subsection{Accuracy of Stabilized Nystr\"om}
In this section, we provide a bound for the accuracy of the stabilized version of the Nystr\"om method for SPSD matrices under the assumption that $S$ is a locally max-vol subsampling matrix.
The stabilized version of the Nystr\"om method is given by
\begin{equation*}
    A_{N}^{(\epsilon)} = AS(S^T\!AS)_{\epsilon}^\dagger S^T\! A,
\end{equation*} where $S\in \mathbb{R}^{n\times r}$ is a locally max-vol subsampling matrix for $A$ with $r$ as the target rank, and its associated error is given by
\begin{equation*}
    A-A_{N}^{(\epsilon)} = (I-AS(S^T\!AS)_{\epsilon}^\dagger S^T)A = (I-P_{AS,S}^\epsilon)A,
\end{equation*}
where we set $P^\epsilon_{AS,S}:=AS(S^T\!AS)_{\epsilon}^\dagger S^T$. 
Recall that $P$ is an (oblique) projector if and only if $P^2 = P$ (and $P$ is nonsymmetric). To verify that $P^\epsilon_{AS,S}$ is an oblique projector we compute
\begin{align*}
({P^{\epsilon}_{AS,S}})^2 &=  AS(S^T\!AS)_{\epsilon}^\dagger S^T\! AS(S^T\!AS)_{\epsilon}^\dagger S^T \\
& = AS(S^T\!AS)_{\epsilon}^\dagger (S^T AS)_\epsilon(S^T\!AS)_{\epsilon}^\dagger S^T\\
&= AS(S^T\!AS)_{\epsilon}^\dagger S^T = P^\epsilon_{AS,S}.
\end{align*}

We begin with four lemmas that are important for carrying out our analysis in \Cref{thm:sn}. First, \Cref{lem:max-volineq} proves an important inequality for a locally max-vol subsampling matrix $S\in 
\mathbb{R}^{n\times r}$ and \Cref{lem:relerr} proves an error bound $\|{(I-(AS)(AS)^\dagger)A}\|_F$ when locally max-vol subsampling matrix $S$ is used. \Cref{lem:max-vol} shows that the projector $P_{AS,S}^\epsilon$ is bounded by a small function involving $n$ and $r$ and finally, in \Cref{lem:eps_proj} we prove that $P_{AS,S}^\epsilon$ projects $AS$ to a slightly perturbed space of $AS$.

\begin{lemma} \label{lem:max-volineq}
    Let $S\in \mathbb{R}^{n\times r}$ be a locally max-vol subsampling matrix for a SPSD matrix $A \in \Re^{n\times n}$ corresponding to the indices $J$. Then
    \begin{equation*}
        \sigma_{\min}(S^T\!Q) \geq \frac{1}{\sqrt{1+r(n-r)}} \geq \frac{1}{\sqrt{nr}}
    \end{equation*} where $Q$ is the orthonormal factor in the thin QR decomposition of $AS$.
\end{lemma}
\begin{proof}
     We first show that $J$ is a locally max-vol index for $AS$, i.e., $|\mathrm{det}(A(\hat{J},J))| \leq |\mathrm{det}(A(J,J))|$ for all $\hat{J} \subseteq [n]$ with $|\hat{J}| = |J|$ and $|J\cap \hat{J}| = r-1$. Let $\hat{J} \subseteq [n]$ such that $|\hat{J}| = |J|$ and $|J\cap \hat{J}| = r-1$. Then by the same argument as \cite[Theorem 1]{cortinovis2020maximum},\footnote{The argument we refer to is only present in the arXiv version; see \url{https://arxiv.org/pdf/1902.02283}}
     \begin{equation*}
         |\mathrm{det}(A(\hat{J},J))|^2 \leq \mathrm{det}(A(J,J)) \cdot \mathrm{det}(A(\hat{J},\hat{J})).
     \end{equation*}
     Now since $J$ is a locally max-vol index for $A$, $\mathrm{det}(A(J,J) \geq \mathrm{det}(A(\hat{J},\hat{J})$, which implies that $|\mathrm{det}(A(\hat{J},J))| \leq \mathrm{det}(A(J,J))$. Therefore, $J$ is a locally max-vol index for $AS$ and since 
     \begin{equation*}
         \det(Q(J,:)) = \det(A(J,J))/\det(R) \geq |\det(A(\hat{J},J))|/\det(R) = \det(Q(\hat{J},:))
     \end{equation*} for any $\hat{J} \subseteq [n]$ with $|\hat{J}| = |J|$ and $|J\cap \hat{J}| = r-1$, $J$ is also a locally max-vol index for $Q$. Now by \cite[Lemma 2.1]{GoreinovTyrtyshinikovZamarashkin1997}, we obtain
     \begin{equation*}
         \sigma_{\min}(S^T\!Q) = \sigma_{\min}(Q(J,:)) \geq \frac{1}{\sqrt{1+r(n-r)}},
     \end{equation*} and the lemma follows by noting that $nr\geq 1+r(n-r)$ for $r\geq 1$.
\end{proof}


In the analysis that follows, we repeatedly use \Cref{lem:max-volineq}. Each invocation of the inequality introduces a factor of $\sqrt{nr}$, so the error bounds in \Cref{lem:max-vol}, \Cref{lem:eps_proj} and \Cref{thm:sn} carry a polynomial prefactor that depends on $\sqrt{n}$ and $\sqrt{r}$. The classical worst-case bound $\sigma_{\min}(S^TQ) \geq 1/\sqrt{1+r(n-r)}$ \cite{GoreinovTyrtyshinikovZamarashkin1997} is theoretically tight, but it is usually quite conservative for most matrices \cite{Massei2022}. In practice, one can check the conditioning of a chosen index set $J$ by computing the QR factorization of $AS = A(:,J) = QR$ and computing the singular values of $S^TQ$. This costs $\text{O}(nr^2+r^3)$ operations. Numerical experiments typically show that the smallest singular value of $S^TQ$ is much larger than $1/\sqrt{nr}$, which confirms that the $\sqrt{nr}$ factors are worst-case artifacts \cite{DongMartinsson2023}. We note that in classical stability analysis~\cite{HighamStabilityBook}, low-degree polynomial factors in $n$ are regarded as insignificant when multiplied by the unit roundoff $u$ or a comparably small quantity. We follow this convention in this paper, even though in modern data science applications $n$ can be large enough to make terms like $n^2u$ potentially large.

\begin{lemma} \label{lem:relerr}
    Let $S\in \mathbb{R}^{n\times r}$ be a locally max-vol subsampling matrix for a SPSD matrix $A$. Then
    \begin{equation*}
        \|{(I-(AS)(AS)^\dagger)A}\|_2 \leq n(1+r) \sigma_{r+1}(A).
    \end{equation*}
\end{lemma}
\begin{proof}
   A locally max-vol subsampling matrix satisfies the classical bound 
\cite{Massei2022,goreinov2001maximal}
\begin{equation*}
    \|A - AS(S^T\!AS)^{-1}S^T\!A\|_{\max} \leq (r+1)\,\sigma_{r+1}(A),
\end{equation*}
where $\|\cdot\|_{\max}$ denotes the maximum magnitude among the entries 
of the matrix argument. 

    Now consider the following optimization problem
    \begin{equation} \label{eq:optprob}
        \min_{C} \|{A-AS C}\|_F.
    \end{equation}  Since the optimization problem \eqref{eq:optprob} has the solution $(AS)^\dagger A=\arg\min_{C} \|{A-AS C}\|_F$, we have
    \begin{equation*}
        \|{A-AS(AS)^\dagger A}\|_F = \min_{C} \|{A-AS C}\|_F \leq \|{A-AS(S^TAS)^{-1}S^T\!A}\|_F.
    \end{equation*}
    Therefore, putting these together we obtain
    \begin{align*}
        \|{(I-(AS)(AS)^\dagger)A}\|_2 &\leq \|{A-AS(AS)^\dagger A}\|_F \leq \|{A-AS(S^T\!AS)^{-1}S^T\!A}\|_F \\
        &\leq n \|{A-AS(S^T\!AS)^{-1}S^T\!A}\|_{\max}\leq n(r+1) \sigma_{r+1}(A),
    \end{align*} where in the penultimate inequality, we used $\|{B}\|_F \leq n \|{B}\|_{\max}$ for $B\in \mathbb{R}^{n\times n}$.
\end{proof}
A recent refinement by \cite{allen2024maximal} slightly sharpens the 
dependence on the singular values, and also establishes guarantees 
for approximate max-vol subsampling. For the purposes of this work, we retain the classical bound in order to keep the presentation simple.

\begin{lemma} \label{lem:max-vol}
Given a SPSD matrix $A\in\mathbb{R}^{n\times n}$, let $S\in \mathbb{R}^{n\times r}$ be a locally max-vol subsampling matrix for $A$ and assume that $r\leq \mathrm{rank}(A)$. 
Then
\begin{equation*}
    \|P_{AS,S}^\epsilon\|_2 = \|AS(S^T\!AS)_{\epsilon}^\dagger\|_2\leq \sqrt{nr}.
\end{equation*}
\end{lemma}
\begin{proof}
Let $AS = QR$ be the thin QR decomposition of $AS$. Then
     \begin{align*}
\|AS(S^T\!A S)_\epsilon^\dagger\|_2 & = 
\|QR(S^T\! AS)_\epsilon^\dagger\|_2 =\|R(S^TAS)_\epsilon^\dagger\|_2=\|(S^T\!Q)^\dagger (S^T\!Q)R(S^T\!AS)_\epsilon^\dagger\|_2\\
&\leq \|(S^T\!Q)^\dagger\|_2 \|(S^T\!AS)(S^T\!AS)_\epsilon^\dagger\|_2=\|(S^T\!Q)^\dagger\|_2, \leq \sqrt{nr}
  \end{align*} where \Cref{lem:max-volineq} was used for the last inequality.
\end{proof}

\begin{lemma} \label{lem:eps_proj}
    With $P^{\epsilon}_{AS, S} = AS (S^T\!AS)_{\epsilon}^\dagger S^T$ as in \Cref{lem:max-vol}
\begin{equation*}
    P^{\epsilon}_{AS, S}AS = AS + E_\epsilon,
\end{equation*} where $\|{E_\epsilon}\|_2 \leq 2\epsilon \sqrt{nr}$.
\end{lemma}
\begin{proof}
We first notice that there exists $A_\epsilon = A+E$ with $\|{E}\|_2 \leq \epsilon$ such that
\begin{equation*}
    (S^T\!AS)_\epsilon = S^T\!A_{\epsilon}S.
\end{equation*} Indeed if  $S^T\!AS = U_1 \Sigma_1 U_1^T + U_2 \Sigma_2 U_2^T$ is the SVD, with $\|\Sigma_2\|_2\leq \epsilon$ , it is sufficient to set $A_\epsilon = A - S U_2\Sigma_2 U_2^T S^{T}$. Then we get $S^T\!A_\epsilon S = [U_1,U_2] \begin{bmatrix}
    \Sigma_1 & 0 \\ 0 & 0
\end{bmatrix}[U_1,U_2]^T$ as the SVD of $S^T\!A_\epsilon S$.
Hence we have $P^\epsilon_{AS,S}AS=AS(S^T\!A S)_\epsilon^\dagger (S^T\!AS)=AS(S^T\! A_\epsilon S)^\dagger (S^T\!AS)$. Now 
\begin{align*}
P_{AS,S}^\epsilon AS&=AS(S^T\! A_\epsilon S)^\dagger (S^T\!AS)
=AS(S^T\! A_\epsilon S)^\dagger (S^T\!(A_\epsilon-E)S)\\
&=AS(S^T\! A_\epsilon S)^\dagger (S^T\! A_\epsilon S)-AS(S^T\! A_\epsilon S)^\dagger S^T\!ES,
  \end{align*}
where the second term in the final equality satisfies 
\begin{align*}
    \|{AS(S^T\! A_\epsilon S)^\dagger S^T\!ES}\|_2 \leq  \|{AS(S^T \!A_\epsilon S)^\dagger}\|_2 \|{S^T\!ES}\|_2 \leq \epsilon  \sqrt{nr} 
\end{align*} by \Cref{lem:max-vol}. Then
\begin{align*}
    AS(S^T\! A_\epsilon S)^\dagger (S^T\! A_\epsilon S) &= ASU_1 U_1^T = AS -ASU_2U_2^T
\end{align*} and
\begin{align*}
    \|{ASU_2 U_2^T}\|_2 & = \|QR U_2U_2^T\|_2=\|R U_2U_2^T\|_2 = \|{(S^T\!Q)^\dagger S^T\!QR U_2 U_2^T}\|_2 \\
     &= \|{(S^T\!Q)^\dagger S^T\! A S U_2 U_2^T}\|_2 \leq \|{(S^T\!Q)^\dagger}\|_2 \|{S^T\!(A_\epsilon-E)S U_2 U_2^T}\|_2 \\
    &\leq \sqrt{nr} \|{-S^T\!ES U_2 U_2^T}\|_2 \leq \epsilon \sqrt{nr},  
\end{align*} where \Cref{lem:max-volineq} and $S^T\!A_{\epsilon} S \tilde{U}_2  = 0$ was used in the penultimate inequality. The result follows by setting
\begin{equation*}
    E_\epsilon:= -AS U_2 U_2^T - AS(S^T\! A_\epsilon S)^\dagger S^T\!ES
\end{equation*}
and noting that
\begin{equation*}
    P_{AS,S}^\epsilon AS = AS-AS U_2 U_2^T - AS(S^T A_\epsilon S)^\dagger S^TES  = AS+ E_\epsilon,
\end{equation*} where $\|E_{\epsilon}\|_2 \leq 2\epsilon\sqrt{nr}$.
\end{proof}

\noindent We now combine the lemmas to analyze the accuracy of the stabilized Nystr\"om method in exact arithmetic.

\begin{theorem}[Accuracy of Stabilized Nystr\"om] \label{thm:sn}
    Let $A\in \mathbb{R}^{n\times n}$ be a SPSD matrix and $S\in \mathbb{R}^{n\times r}$ be a locally max-vol subsampling matrix for $A$. 
    Then
    \begin{equation*}
        \|A-A_{N}^{(\epsilon)}\|_2 \leq n^2r(r+1) \sigma_{r+1}(A) + 2nr\epsilon.
    \end{equation*}
\end{theorem}
\begin{proof}
We first note that, by \Cref{lem:eps_proj},
\begin{align*}
    A-A_{N}^{(\epsilon)} = (I-P_{AS,S}^\epsilon)A
    = (I-P_{AS,S}^\epsilon)A(I-SM_S)+ E_\epsilon M_S
\end{align*} for any matrix $M_S \in \mathbb{R}^{r\times n}$ and $E_\epsilon$ is as in \Cref{lem:eps_proj}, satisfying $\|{E_\epsilon}\|_2 \leq 2\epsilon \sqrt{nr}$. Now choose $M_S = (Q^T\!S)^\dagger Q^T$ where $Q$ is the orthonormal matrix in the thin QR decomposition of $AS$. We get
\begin{align*}
    (I-P_{AS,S}^\epsilon)A(I-SM_S) &= (I-P_{AS,S}^\epsilon)A(I-S(Q^T\!S)^\dagger Q^T) \\
    &= (I-P_{AS,S}^\epsilon)A(I-QQ^T)(I-S(Q^T\!S)^\dagger Q^T)
\end{align*} since $Q^T\!S$ is non-singular by \Cref{lem:max-volineq} and
\begin{equation*}
    \|{E_\epsilon M_S}\|_2 \leq \|{E_\epsilon}\|_2 \|{(Q^T\!S)^\dagger}\|_2 \leq 2\epsilon \sqrt{nr} \sqrt{nr} = 2\epsilon nr
\end{equation*} by \Cref{lem:max-volineq}. Finally, we get
\begin{align*}
    \|A-A_{N}^{(\epsilon)}\|_2 &\leq \|{(I-P_{AS,S}^\epsilon)A(I-QQ^T)(I-S(Q^T\!S)^\dagger Q^T)}\|_2 + \|{E_\epsilon M_S}\|_2 \\
    &\leq \|{I-P_{AS,S}^\epsilon}\|_2 \|{A(I-QQ^T)}\|_2\|{I-S(Q^T\!S)^\dagger Q^T}\|_2 + 2nr\epsilon \\
    &\leq \|{P_{AS,S}^\epsilon}\|_2 \|{A(I-QQ^T)}\|_2\|{S(Q^T\!S)^\dagger Q^T}\|_2 + 2nr\epsilon \\
    &\leq \sqrt{nr} \|{A(I-QQ^T)}\|_2 \sqrt{nr} + 2nr\epsilon  \\
    &\leq n^2r(r+1)\sigma_{r+1}(A) + 2nr\epsilon,
\end{align*} where from the second to the third line we used the fact that any oblique projector $P$ satisfies $\|I-P\|_2=\|P\|_2$ and in the last two inequalities we used \Cref{lem:max-volineq,lem:max-vol,lem:relerr}.
\end{proof}

\section{Numerical Stability of SN} \label{sec:stability of MLN}

So far, we have analyzed SN in exact arithmetic, showing that it closely preserves the accuracy of the Nystr\"om algorithm. Here we address the problem of the stability of the algorithm; that is, we provide an actual implementation of the algorithm and show that even taking into account round-off errors, the computed approximant $fl(AS(S^T\!AS)^\dagger_\epsilon S^T\!A)$ has error $\|A-fl(A_{SN})\|_F$ comparable to $A-A_{N}^{(\epsilon)}$. We consider the following implementation: we first compute $AS$ and $S^T\!AS$. Next, we compute a truncated Cholesky decomposition $R_\epsilon^TR_\epsilon = (S^T\!AS)_\epsilon$. Finally, we obtain the low-rank factor $ASR_{\epsilon}^{\dagger}$.

Throughout the analysis, we let $u$ denote the unit roundoff, and use the symbol $\mach$ 
to represent either a scalar or a matrix whose norm is bounded 
by at most a low-degree polynomial in $n$ and $r$, multiplied 
by $u$. In particular, $\|\mach\|_F$ does not grow 
exponentially with $n$, nor does it scale inversely with small 
quantities such as $\epsilon$ or $\sigma_r(A)$.
Similarly, we distinguish between the notation 
$\text{O}(1)$ to denote a quantity bounded in magnitude by a constant 
$c$, where $c$ is independent of the problem size and $\O(1)$ which suppresses low-degree polynomial in $n$ and $r$. 
While this may appear to be a simplification, it is a standard 
convention in stability analyses (see, e.g., 
\cite{nakatsukasa2012backward}). The precise value of $\mach$ may vary from one occurrence 
to another, but we always assume that $\epsilon \gg \mach$. 
Unless explicitly stated otherwise, we normalize so that $\|A\|_2=1$.

\begin{proposition}\label{prop:computed_version}
Let $S\in\mathbb{R}^{n\times r}\ (n>r)$ be any orthonormal matrix.
There exist matrices $\Ah = A +\Delta \Ah$ and $\At = A + \Delta \At$, with $\|\Delta \Ah\|_2,\|\Delta \At\|_2\leq \mach$ such that
\begin{itemize}
    \item $fl(AS) = \Ah S$;
    \item $fl(S^T\!AS) = S^T\! \At S$.
\end{itemize}
\end{proposition}
These statements are perhaps trivially true when $S$ is a subsampling matrix, often with $\Ah=A$ and $\At=A$. We state a more general version in case there are errors incurred in computing with $S$.

\begin{proof}
    For any matrix $A$ and $B$, $\|fl(AB)-AB\|_2 \leq \|A\|_2\|B\|_2\mach$ \cite[Section 3.5]{HighamStabilityBook} and since $\|S\|_2 =1$ and $\|A\|_2=1$
    
    \[fl(AS) = AS + \mach = (A+\mach S^T) S =: \Ah S\]
    and
    
    \[fl(S^T\!AS) = S^T\!AS + \mach = S^T\!(A+S\mach S^T) S =: S^T\!\At S.\]
\end{proof}
We will also assume that the computed truncated Cholesky factor $\widehat{R}_\epsilon = fl(\text{chol}((S^T\!\At S)_\epsilon))$ satisfies
\begin{equation}\label{eq:cpcholesky}
    \widehat{R}_\epsilon ^T \widehat{R}_\epsilon = (S^T\!\At S)_\epsilon + \mach,
\end{equation}
where $\widehat{R}_\epsilon \in\mathbb{R}^{\hat{r}\times r}$, with $\hat{r}\leq r$, is numerically full-rank. In fact, by applying the Cholesky algorithm with complete pivoting, one can establish the pessimistic bound \cite[Theorem 10.14]{HighamStabilityBook}
$$
\|(S^T\! \At S)_\epsilon -\widehat{R}_\epsilon^T \widehat{R}_\epsilon\|_2\leq 4^{\hat{r}}\mach,
$$
even though the dimensions of the matrix are relatively small, there is still an exponential dependence in $\hat{r}$, which we could not ignore in general. However, as explained in \cite[Theorem 10.14]{HighamStabilityBook}, this bound tends to be overly pessimistic in practice, and for our purposes, we can safely disregard it.

We now prove the floating-point version of the \Cref{lem:max-vol,lem:eps_proj}.

\begin{lemma} \label{lem:max-vol_fl}
    Let $\Ah$ and $\At$ be as in \Cref{prop:computed_version}, and suppose $S\in\mathbb{R}^{n\times r}$ is a locally max-vol subsampling matrix for $A$.
    then
    \begin{equation}
       \|\Ah S (S^T\! \At S)^\dagger_\epsilon\|_2\leq 1 + 2\sqrt{nr}.
    \end{equation}
\end{lemma}

\begin{proof}
    Let $AS = Q R$ be the thin QR decomposition of $A S$. Then
     \begin{align*}
\|\Ah S(S^T\!\At S)_\epsilon^\dagger\|_2 & = 
\|(A +\mach) S(S^T\!\At_\epsilon S)^\dagger\|_2 \leq \|A  S(S^T\!\At_\epsilon S)^\dagger\|_2 +\|\mach S(S^T\!\At_\epsilon S)^\dagger\|_2\\
& \leq \|Q R(S^T \!\At_\epsilon S)^\dagger\|_2 + 1  =\|R(S^T\! \At_\epsilon S)^\dagger\|_2+ 1\\
&=\|(S^T\!Q)^{-1}(S^T\!Q )R(S^T\!\At_\epsilon S)^\dagger\|_2+ 1\leq \|(S^T\!Q)^{-1}\|_2 \|(S^T\!Q)R(S^T\!\At_\epsilon S)^\dagger\|_2+ 1\\
&= \|(S^T\!Q )^{-1}\|_2 \|(S^T\!(\At_\epsilon +\mach) S)(S^T\! \At_\epsilon S)^\dagger\|_2+ 1\\
&\leq \|(S^T\!Q )^{-1}\|_2 \left(\|(S^T\!\At_\epsilon S)(S^T\! \At_\epsilon S)^\dagger\|_2+\|(S^T\!\mach S)(S^T \!\At_\epsilon S)^\dagger\|_2\right)+ 1 \\
&\leq 1 + 2\|(S^T\!Q)^{-1}\|_2\leq 1 + 2\sqrt{nr}. 
  \end{align*}
\end{proof}

\begin{lemma} \label{lem:projeqn}
Let $\Ah$ and $\At$ be as in \Cref{prop:computed_version}, then
\begin{equation}
    \Ah S (S^T\!\At S)^\dagger_\epsilon (\Ah S)^T\! S = AS + E_\epsilon,
\end{equation}
where $\|E_{\epsilon}\|_2 \leq 4\epsilon\sqrt{nr}$.
\end{lemma}
\begin{proof}
    As in \Cref{lem:eps_proj}, there exists $\At_\epsilon = \At+E$ with $\|{E}\|_2 \leq \epsilon$ such that
\begin{equation*}
    (S^T\!\At S)_\epsilon = S^T\!\At_{\epsilon}S.
\end{equation*} 

Hence we have
\[
\Ah S (S^T\! \At S)^\dagger_\epsilon (\Ah S)^T S=\Ah S(S^T\!\At_\epsilon S)^\dagger (S^T\!\Ah S)=\Ah S(S^T\!\At_\epsilon S)^\dagger (S^T\!(\At_\epsilon-\mach) S).
\] The term $\Ah S(S^T\!\At_\epsilon S)^\dagger S^T\mach  S$, by \Cref{lem:max-vol_fl}, satisfies
\[
\|\Ah S(S^T\!\At_\epsilon S)^\dagger S^T\!\mach  S\|_2\leq  \|\Ah S(S^T\!\At_\epsilon S)^\dagger\|_2 \|S^T\! \mach S\|_2 \leq \mach(1+2\sqrt{nr}).
\]
Hence we focus on the remaining term.
Let $S^T\!\At_\epsilon S = [\tilde{U}_1,\tilde{U}_2] \begin{bmatrix}
    \tilde{\Sigma}_1 & 0 \\ 0 & 0
\end{bmatrix}[\tilde{U}_1,\tilde{U}_2]^T$ be the SVD of $S^T\At_\epsilon S$ where $[\tilde{U}_1,\tilde{U}_2] \in \mathbb{R}^{r\times r}$ is an orthogonal matrix and $\tilde{\Sigma}_1$ contains singular values of $S^T\At_\epsilon S$ greater than $\epsilon$. Then
\begin{align*}
    \Ah S(S^T\! \At_\epsilon S)^\dagger (S^T\! \At_\epsilon S) &= \Ah S\tilde{U}_1 \tilde{U}_1^T = A S -A S\tilde{U}_2\tilde{U}_2^T + \mach,
\end{align*} where
\begin{align*}
    \|{A S\tilde{U}_2\tilde{U}_2^T}\|_2 &\leq \|{(S^T\!Q)^\dagger S^T\! A S \tilde{U}_2\tilde{U}_2^T}\|_2 \leq \|{(S^T\!Q)^\dagger}\|_2 \|{S^T\!(\At_\epsilon- E + \mach)S\tilde{U}_2\tilde{U}_2^T}\|_2 \\
    &\leq \sqrt{nr} \|{S^T(E-\mach) S\tilde{U}_2\tilde{U}_2^T}\|_2 \leq 2\epsilon \sqrt{nr}  
\end{align*} where \Cref{lem:max-volineq} and $S^T\!\At_{\epsilon} S \tilde{U}_2  = 0$ was used in the penultimate inequality. The result follows by putting everything together.
\end{proof}

We now analyze the quantity  $fl(AS(S^T\!AS)^\dagger_\epsilon S^T\!A)$.
 Let us denote the $i$th row of a matrix $Z$
by $[Z]_i$ and the element in position $(i, j)$ by $[Z]_{ij}$.

\begin{lemma}\label{lem:floating_point_ij}
    Let $fl([AS(S^T\!AS)^\dagger_\epsilon S^T\!A]_{ij})$ be obtained by first computing an $\epsilon$-truncated Cholesky on $fl(S^T\!AS)$, giving $\widehat{R}_\epsilon$ satisfying \eqref{eq:cpcholesky},
    and then solving linear systems
    of the form $fl([AS]_k\widehat{R}_{\epsilon}^\dagger)$, for $k = i,j$, using a backward stable underdetermined linear solver.
    Then
    \begin{equation}
        fl\left([AS(S^T\!AS)^\dagger_\epsilon S^T\!A]_{ij}\right)=[AS+\mach]_i(S^T\!AS + \mach)^\dagger_{\widetilde{\epsilon}}[AS+\mach]_j^T,
    \end{equation}
    where $\widetilde{\epsilon} \geq \epsilon-\mach$.
\end{lemma}

\begin{proof}
Recall that $fl(AS)= AS+\mach$. Thus, the overdetermined least-squares linear systems that we need to solve are of the form,
\begin{equation}
    \min_{x}\left\|x\widehat{R}_{\epsilon} - ([AS]_k+\mach)\right\|_2 
\end{equation} for $k = i,j$. Since the matrix $\widehat{R}_{\epsilon}$ is numerically full-rank, by \cite[Theorem 20.3]{HighamStabilityBook}, 
the computed solution satisfies
 \[
fl\left(\left([AS]_k+\mach\right)\widehat{R}_{\epsilon}^\dagger\right) = \left(\left([AS]_k+\mach\right)+\mach\right)\left(\widehat{R}_{\epsilon}+\mach\right)^\dagger = \left([AS]_k+\mach\right)\left(\widehat{R}_{\epsilon}+\mach\right)^\dagger.
\]
It remains to multiply $([AS]_i+\mach)\left(\widehat{R}_{\epsilon}+\mach\right)^\dagger$ and $\left(\widehat{R}_{\epsilon}+\mach\right)^{\dagger T}([AS]_j+\mach)^T$, that is
\begin{align*}
&fl\left(([AS]_i+\mach)\left(\widehat{R}_{\epsilon}+\mach\right)^\dagger \left(\widehat{R}_{\epsilon}+\mach\right)^{\dagger T}([AS]_j+\mach)^T\right) = \\
&\qquad = ([AS]_i+\mach)\left(\widehat{R}_{\epsilon}+\mach\right)^\dagger \left(\widehat{R}_{\epsilon}+\mach\right)^{\dagger T}([AS]_j+\mach)^T \\
&\qquad + \mach\left\| ([AS]_i+\mach)\left(\widehat{R}_{\epsilon}+\mach\right)^\dagger\right\|_2 \left\|([AS]_j+\mach)\left(\widehat{R}_{\epsilon}+\mach\right)^\dagger \right\|_2.
\end{align*}

We now prove that 
\begin{equation}\label{eq:tamed}
\left\| ([AS]_k+\mach)\left(\widehat{R}_{\epsilon}+\mach\right)^\dagger\right\|_2= \text{\O}(1)
\end{equation}
for $k=1, \dots, n$. Indeed

\begin{align*}
\left\| ([AS]+\mach)\left(\widehat{R}_{\epsilon}+\mach\right)^\dagger \right\|_2 &= \left\|(S^T\!Q)^{-1}\right\|_2 \left\|(S^T\!AS+\mach)\left(\widehat{R}_{\epsilon}+\mach\right)^\dagger\right\|_2\\
&= \text{\O}(1) \left\|\left(\widehat{R}_\epsilon^T \widehat{R}_\epsilon + \epsilon + \mach\right) \left(\widehat{R}_{\epsilon}+\mach\right)^\dagger\right\|_2 \\ 
& \leq \text{\O}(1) \left(\left\| \left(\widehat{R}_\epsilon^T \widehat{R}_\epsilon \right)  \left(\widehat{R}_\epsilon+\mach\right)^\dagger\right\|_2+\left\|(\epsilon + \mach) \left(\widehat{R}_\epsilon+\mach\right)^\dagger\right\|_2\right).
\end{align*}
By Weyl's inequality,
\[
\sigma_k((S^T\!AS)_\epsilon)-\mach \leq \sigma_k((S^T\!AS)_\epsilon+\mach) \leq \sigma_k((S^T\!AS)_\epsilon) +\mach.
\]
Since $\sigma_{k}(\widehat{R}_\epsilon)= \sqrt{\sigma_{k}((S^T\!AS)_\epsilon +\mach})$, we obtain
\[
\sigma_{\min}(\widehat{R}_\epsilon-\mach)\geq \sigma_{\min}(\widehat{R}_\epsilon)-\mach\geq \sqrt{\epsilon-\mach}-\mach.
\]
Thus 
\begin{align*}
    &\left\| \left(\widehat{R}_\epsilon^T \widehat{R}_\epsilon \right) \left(\widehat{R}_\epsilon+\mach\right)^\dagger\right\|_2\leq \|\widehat{R}_\epsilon\|_2 \|\widehat{R}_\epsilon \left(\widehat{R}_\epsilon+\mach\right)^\dagger\|_2\\
    &\qquad\leq  \|\widehat{R}_\epsilon\|_2 \|(\widehat{R}_\epsilon +\mach-\mach)\left(\widehat{R}_\epsilon+\mach\right)^\dagger\|_2= \|\widehat{R}_\epsilon\|_2 \|I- \mach\left(\widehat{R}_\epsilon+\mach\right)^\dagger\|_2 = \text{O}(1)
    \end{align*}
    and
    \begin{align*}
        \|(\epsilon + \mach) \left(\widehat{R}_\epsilon+\mach\right)^\dagger\|_2= \text{O}(1), 
    \end{align*} completing the proof of \eqref{eq:tamed}.

Finally, we prove that there exists a $\widetilde \epsilon \geq \epsilon - \mach$ such that
\[\left(\widehat{R}_{\epsilon}+\mach\right)^\dagger \left(\widehat{R}_{\epsilon}+\mach\right)^{\dagger T}  = (S^T\!AS+\mach)^\dagger_{\widetilde{\epsilon}}\] .\\

Recall that $\widehat{R}_{\epsilon}\in \mathbb{R}^{\hat{r}\times r}$, with $\hat{r}\leq r$ satisfies $\sigma_{k}(\widehat{R}_{\epsilon}^T\widehat{R}_{\epsilon})\geq \epsilon$ for all $k\leq \hat{r}$. Thus by Weyl's inequality, for all $k\leq \hat{r}$
\[
\sigma_{k}\left(\left(\widehat{R}_{\epsilon}+\mach\right)^T\left(\widehat{R}_{\epsilon}+\mach\right)\right) = \sigma_{k}\left(\widehat{R}_{\epsilon}^T\widehat{R}_{\epsilon}+\mach\right)\geq \epsilon - \mach.
\]
At the same time, being product of two rank $\hat{r}$ matrices,  for all $k>\hat{r}$
\[\sigma_{k}\left(\left(\widehat{R}_{\epsilon}+\mach\right)^T\left(\widehat{R}_{\epsilon}+\mach\right)\right)=0.\]
Hence
\begin{align*}
\left(\widehat{R}_{\epsilon}+\mach\right)^\dagger \left(\left(\widehat{R}_{\epsilon}+\mach\right)^{\dagger}\right)^T  &= \left(\left(\widehat{R}_{\epsilon}+\mach\right)^T \left(\widehat{R}_{\epsilon}+\mach\right)\right)^\dagger \\ &=\left(\left(\widehat{R}_{\epsilon}+\mach\right)^T \left(\widehat{R}_{\epsilon}+\mach\right)\right)_{\widetilde{\epsilon}}^\dagger \\ &= \left(S^T\! A S + \mach\right)^\dagger_{\widetilde{\epsilon}}.
\end{align*}

Combining all the steps, we arrive at the desired result.
\end{proof}

\noindent We now state the main stability result of SN. 

\begin{theorem}[Nystr\"om's stability]\label{thm:floating_point}
    Suppose we compute $AS(S^T\!AS)^\dagger_\epsilon S^T\! A$ as in \Cref{lem:floating_point_ij}, 
    then
    \[
    \|A-fl(AS(S^T\!AS)^\dagger_\epsilon S^T\! A)\|_F\leq n\sqrt{nr}\left(4n\sqrt{nr}(r+1)\sigma_{r+1}(A)+\epsilon + \mach \right) + \mach.
    \]
    \end{theorem}
\begin{proof}
    For shorthand let $A_{N,\epsilon} = (AS+\mach)(S^T\!AS +\mach)_{\widetilde\epsilon}^\dagger (S^T\! A+\mach)$.
    By \Cref{lem:projeqn}
    \begin{equation*}
        A_{N,\epsilon}S = AS +\widetilde\epsilon.
    \end{equation*} Following the proof of \Cref{thm:sn}, we obtain
    \begin{align*}
        A-A_{N,\epsilon} &= (A-A_{N,\epsilon})(I-S(Q^TS)^\dagger Q^T) +{\widetilde\epsilon}(Q^T\!S)^\dagger Q^T \\
        &= (A-A_{N,\epsilon})(I-QQ^T)(I-S(Q^T\!S)^\dagger Q^T) +\widetilde\epsilon \sqrt{nr}
    \end{align*}since $\|{(Q^T\!S)^\dagger}\|_2 \leq \sqrt{nr}$ by \Cref{lem:max-volineq}. Now note that
    \begin{align*}
        A-A_{N,\epsilon} &=  A-(AS+\mach)(S^T\!AS +\mach)_{\widetilde\epsilon}^\dagger (S^T\! A+\mach) \\
        &= (I-(AS+\mach)(S^T\!AS +\mach)_{\widetilde\epsilon}^\dagger S^T)A +(AS+\mach)(S^T\!AS +\mach)_{\widetilde\epsilon}^\dagger\mach \\
        &= \left(I-(AS+\mach)(S^T\!AS +\mach)_{\widetilde\epsilon}^\dagger S^T\right)A +\mach
    \end{align*} since $\|{(AS+\mach)(S^T\!AS +\mach)_\epsilon^\dagger}\|_2 \leq 1+2\sqrt{nr}$ by \Cref{lem:max-vol_fl}. Therefore, 
    \begin{equation*} 
        A-A_{N,\epsilon} = \left(I-(AS+\mach)(S^T\!AS +\mach)_{\widetilde\epsilon}^\dagger S^T\right)A(I-QQ^T)(I-S(Q^T\!S)^\dagger Q^T) +\mach + {\widetilde\epsilon} \sqrt{nr}.
    \end{equation*} 
Hence
    \begin{align*}
        \|{A-A_{N,\epsilon}}\|_2 &\leq \|{\left(I-(AS+\mach)(S^T\!AS +\mach)_{\widetilde\epsilon}^\dagger S^T\right)}\|_2\|{A(I-QQ^T)}\|_2\|{(I-S(Q^T\!S)^\dagger Q^T)}\|_2\\
        & + \mach +{\widetilde\epsilon} \sqrt{nr}\\
        &\leq \left(1+(1+2\sqrt{nr})\right)(n(r+1)\sigma_{r+1}(A))\sqrt{nr} +{\widetilde\epsilon} \sqrt{nr} +\mach
    \end{align*} by \Cref{lem:max-vol_fl,lem:relerr,lem:max-volineq} respectively, for the first three terms in the final inequality. Therefore,
    \begin{align*}
        \|{A-A_{N,\epsilon}}\|_2 &\leq 4\sqrt{nr}(n(r+1)\sigma_{r+1}(A))\sqrt{nr} +{\widetilde\epsilon} \sqrt{nr}+ \mach\\
        &= (4n\sqrt{nr}(r+1)\sigma_{r+1}(A)+\widetilde\epsilon)\sqrt{nr} + \mach\\
    \end{align*}

In particular
\[
        \left|A_{ij}-\left[fl\left(A_N^{(\epsilon)}\right)\right]_{ij}\right| < (4n\sqrt{nr}(r+1)\sigma_{r+1}(A)+\widetilde\epsilon)\sqrt{nr} + \mach.
\]
    Therefore,
    \begin{align*}
        \|{A-fl\left(A_N^{(\epsilon)}\right)}\|_F & = \sqrt{\sum_{i = 1}^n \sum_{j = 1}^n \left|A_{ij}-\left[fl\left(A_N^{(\epsilon)}\right)\right]_{ij}\right|^2} \\
        &\leq \sqrt{\sum_{i = 1}^n \sum_{j = 1}^n \left(4n\sqrt{nr}(r+1)\sigma_{r+1}(A)+\widetilde\epsilon)\sqrt{nr} + \mach\right)^2} \\
        & = \sqrt{\sum_{i = 1}^n \sum_{j = 1}^n (4n\sqrt{nr}(r+1)\sigma_{r+1}(A)+\widetilde\epsilon)\sqrt{nr})^2+\mach} \\
        &= n\sqrt{nr}\left(4n\sqrt{nr}(r+1)\sigma_{r+1}(A)+\epsilon+\mach \right) + \mach,
    \end{align*} where in the third line we use the assumption that $\sigma_{r+1}(A)\leq \|{A}\|_2 = 1$.
\end{proof}
The inequality in \Cref{thm:floating_point} shows that, when using a locally max-vol indices, the error incurred in computing the \Nystrom approximation in finite-precision arithmetic is bounded---up to a small polynomial factor---by the error obtained in exact arithmetic. This establishes the numerical stability of the method when implemented as described in \Cref{lem:floating_point_ij}.

\section{Numerical Experiments and Implementation Details}
In this section, we illustrate the stability of the Nystr\"om method through numerical experiments. As outlined in \Cref{sec:intro}, we implement the stabilized algorithm described in \Cref{Alg:stableNystrom} by truncating the small eigenvalues in the core matrix via pivoted Cholesky, using the \texttt{cholp} function from Higham's Matrix Computation Toolbox \cite{cholpcode}. Provided the core matrix admits a Cholesky factorization\footnote{In the case when Cholesky fails, the eigendecomposition of the core matrix can be used instead where the negative eigenvalues are set to zero.}, \Cref{Alg:stableNystrom} satisfies the conditions needed for our stability theory so it is numerically stable, whereas \Cref{alg:NaiveNystrom} may not be.

We first compare various different implementations of the Nystr\"om method in \Cref{subsec:implementation} to observe their empirical stability. We then compare the naive Nystr\"om method (\Cref{alg:NaiveNystrom}), stabilized Nystr\"om method (\Cref{Alg:stableNystrom}), and the shifted variant from \cite{CarsonDauvzickaite2024,TroppYurtseverUdellCevher2017spsd} using kernel matrices of various datasets in \Cref{subsec:kernelexp}. In all the experiments, we select the index set $J$ with a strong rank-revealing QR factorization, whose pivoting steps aim to locally maximize the volume\footnote{The MATLAB code for the strong rank-revealing QR is taken from \cite{sRRQRcode}.}, and take $10u\norm{A}_2$ as the shift and truncation parameter where $u$ is the unit roundoff. We use the best rank-$r$ approximation using the truncated SVD (TSVD) as reference and plot the relative error in the Frobenius norm, i.e., if $\hat{A}$ is a rank-$r$ approximation then $\displaystyle\|A-\hat{A}\|_F/\|A\|_F$ is plotted. The experiments were conducted in MATLAB R2024b using double precision arithmetic.

\subsection{Implementation details} \label{subsec:implementation}
The main numerical challenge in the Nystr\"om method lies in forming the pseudoinverse of the core matrix $S^T\!AS$, which can be extremely ill-conditioned. When $A$ is low-rank and $S$ selects a good set of columns/rows, the smallest singular value of $S^T\!AS$ is smaller than $\sigma_r(A)$ by Courant-Fischer theorem, so any instability in the pseudoinverse is amplified. To investigate this issue, we explore five different implementations of the Nystr\"om method in MATLAB:
\begin{enumerate}
    \item \texttt{plain}: $A\approx AS(S^T\!AS)^\dagger S^T\!A$ implemented as in \Cref{alg:NaiveNystrom},
    \item \texttt{shift}: Shifted version implemented as in \cite{CarsonDauvzickaite2024,TroppYurtseverUdellCevher2017spsd},
    \item \texttt{trunc}: Stabilized Nystr\"om method implemented as in \Cref{Alg:stableNystrom},
    \item \texttt{backslash}: implemented as $(AS)*((S^T\!AS)\backslash (AS)^T)$,
    \item \texttt{pinv}: implemented as $(AS) * \texttt{pinv}(S^T\!AS) *(AS)^T$.
\end{enumerate}
We test these five implementations using two symmetric positive semidefinite (SPSD) test matrices:
\begin{enumerate}
    \item \texttt{SNN}: Random sparse non-negative matrices \cite{SorensenEmbree2016,VoroninMartinsson2017}. We use
    \begin{equation*}
        \texttt{SNN} = \sum_{j = 1}^{150} \frac{1}{j} x_j x_j^T + \sum_{j = 151}^{350} \frac{10^{-5}}{j} x_j x_j^T + \sum_{j = 351}^{500} \frac{10^{-10}}{j} x_j x_j^T \in \Re^{1000\times 1000},
    \end{equation*} where $x_j$'s are computed in MATLAB using the command $\texttt{sprandn(1000,1,0.01)}$.
    \item \texttt{ex33}: $1733\times 1733$ matrix from the SuiteSparse Matrix Collection \cite{FloridaDataset2011}.
\end{enumerate}

\begin{figure}[!ht]
\subfloat[\centering \texttt{SNN}]{\label{subfig:impl1}\includegraphics[scale = 0.9]{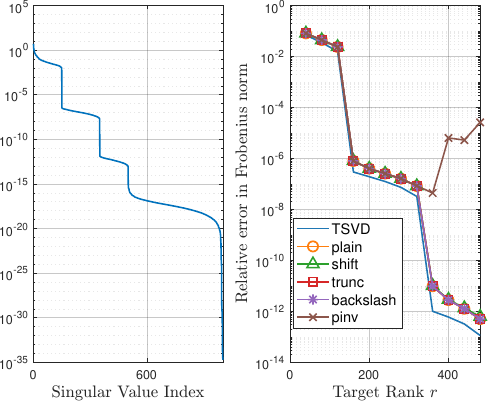}} \hfill
\subfloat[\centering \texttt{ex33}]{\label{subfig:impl2}\includegraphics[scale = 0.9]{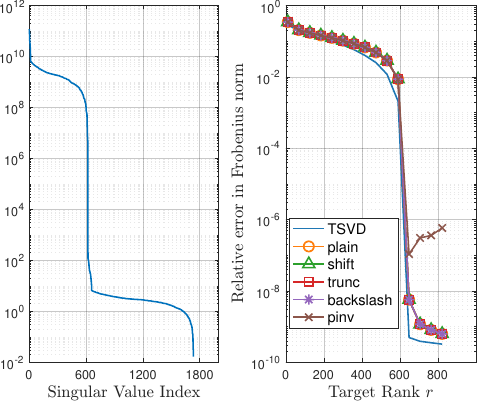}} 
\centering 
\caption{All the different implementations appear to be numerically stable except for the $\texttt{pinv}$ implementation.}
\label{fig:Implementation}
\end{figure}

\Cref{fig:Implementation} summarizes the results. The \texttt{pinv} implementation becomes unstable even at moderate rank values: although the $r$th singular value of $A$ is still sufficiently larger than machine precision ($\approx 2.2\times 10^{-16}$ in double precision), the error spikes and becomes unstable. In contrast, the other four implementations closely follow the truncated SVD error and behaves in a numerically stable manner. We note, however, that in the kernel experiments (see \Cref{subsec:kernelexp}), if a poor set of indices is chosen, the core matrix can become extremely ill-conditioned and the \texttt{backslash} implementation may fail in a similar manner to the \texttt{plain} version. When a well-chosen set of indices is used, the method is surprisingly stable. The MATLAB's backslash solver (\texttt{backslash} implementation) should nevertheless be used with care on numerically rank-deficient underdetermined systems, as it returns a sparse solution based on a pivoting strategy~\cite[\S~2.4]{anderson1995lapack}, which may differ from the minimum-norm solution and therefore might not satisfy the assumptions in our analysis (\Cref{sec:stability of MLN}). We recommend our proposed approach when a theoretical guarantee of stability is desired.

\subsection{Comparison between different implementations using Kernels} \label{subsec:kernelexp}
In this section, we compare four Nystr\"om approximations--the plain Nystr\"om method (\Cref{alg:NaiveNystrom}), the stabilized Nystr\"om method (\Cref{Alg:stableNystrom}), the shifted Nystr\"om method \cite{CarsonDauvzickaite2024,TroppYurtseverUdellCevher2017spsd}, and the backslash implementation--on kernel matrices derived from a range of datasets in LIBSVM \cite{LIBSVM} and the UCI Machine Learning Repository \cite{UCIML}.

The datasets used are summarized in \Cref{tab:datasets}. For each dataset, we draw a subsample of $n = 2000$ points, and normalize each feature to have zero mean and unit variance. We then form an RBF kernel $K(x_i,x_j) = \exp\left(-\frac{\norm{x_i-x_j}_2^2}{2\sigma^2}\right)$ where $\sigma$ is the bandwidth. Larger bandwidths yield smoother kernels with rapid singular-value decay (i.e., lower effective rank), whereas small bandwidths produce kernels of higher rank with larger trailing singular values.

\begin{table}[t]
\centering
\caption{Overview of datasets used in our experiments. Listed are the number of samples $n$, the feature dimension $d$, and the source repository.}
\label{tab:datasets}
\begin{tabular}{lrrl}
\toprule
Dataset & $n$ & $d$ & Source \\
\midrule
\texttt{a9a}                  & 32\,561   & 123   & LIBSVM \\
\texttt{Anuran Calls}         & 7\,195    & 22    & UCI \\
\texttt{cadata}               & 16\,512   & 8     & LIBSVM \\
\texttt{CIFAR10}              & 60\,000   & 3\,072 & \cite{Krizhevsky2009} \\
\texttt{cod-rna}              & 59\,535   & 8     & LIBSVM \\
\texttt{connect-4}            & 54\,045   & 126   & LIBSVM \\
\texttt{covertype}            & 581\,012  & 54    & UCI \\
\texttt{covtype.binary}       & 464\,809  & 54    & LIBSVM \\
\texttt{ijcnn1}               & 49\,990   & 22    & LIBSVM \\
\texttt{phishing}             & 8\,844    & 68    & LIBSVM \\
\texttt{sensit\_vehicle}      & 78\,823   & 100   & LIBSVM \\
\texttt{sensorless}           & 58\,509   & 48    & LIBSVM \\
\texttt{skin\_nonskin}        & 196\,045  & 3     & LIBSVM \\
\texttt{w8a}                  & 49\,749   & 300   & LIBSVM \\
\texttt{YearPredictionMSD}    & 463\,715  & 90    & LIBSVM \\
\bottomrule
\end{tabular}
\end{table}

\begin{figure}[!ht]
\vspace{-0.5cm}
\subfloat[\centering \texttt{ijcnn1}]{\label{subfig:3ijcnn1}\includegraphics[scale = 0.5]{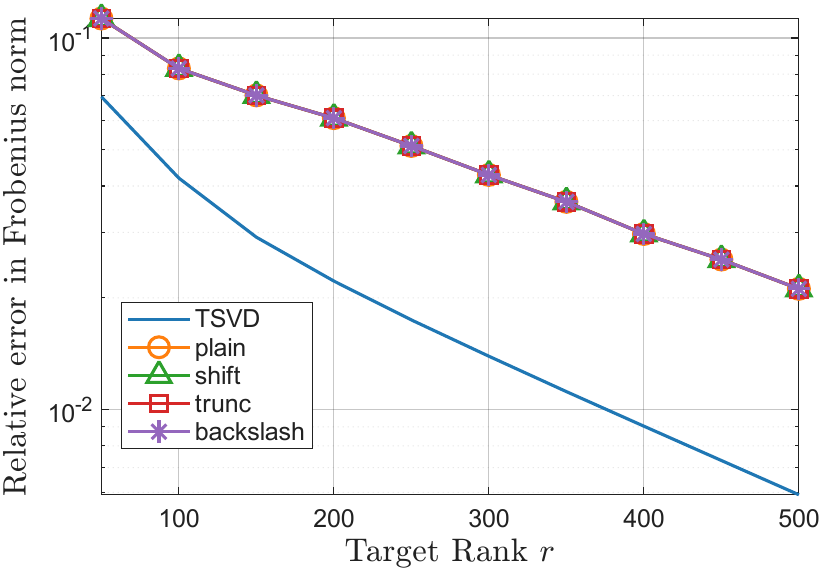}} \hfill
\subfloat[\centering \texttt{skin\_nonskin}]{\label{subfig:3skinnonskin}\includegraphics[scale = 0.5]{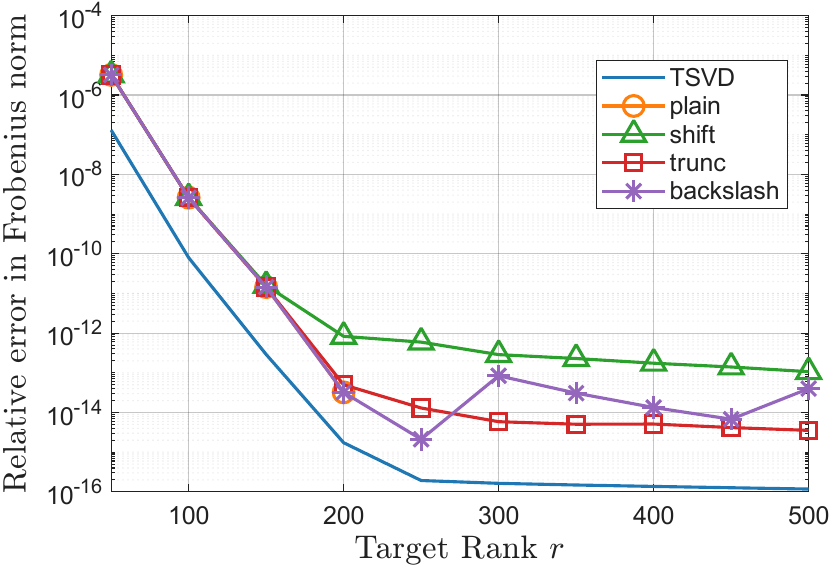}} 
\centering 
\caption{Relative error in Frobenius norm versus target rank $r$ for kernels with $\sigma = 3$. Left (\texttt{ijcnn1}): all four Nystr\"om variants achieve virtually identical accuracy.  Right (\texttt{skin\_nonskin}): the kernel becomes ill‑conditioned as $r$ grows; the plain method fails beyond $r=240$, and the stabilized method outperforms the shifted variant by roughly two orders of magnitude, while the backslash implementation lies in between the two.}
\label{fig:kerneltest1}
\end{figure}

\paragraph{Moderate Bandwidth ($\sigma = 3$)} We conducted experiments on all datasets but show results for two representative datasets in \Cref{fig:kerneltest1}, using a fixed bandwidth of $\sigma = 3$. For \texttt{ijcnn1} (\Cref{subfig:3ijcnn1}), the error curves of the plain, shifted, truncated, and backslash implementations are essentially identical across all ranks. The same pattern is observed on the other datasets not shown in the figure. This occurs because, with this bandwidth, the kernel matrices retain relatively large singular values, and the pseudoinversion of the core matrix remains stable. In contrast, the \texttt{skin\_nonskin} dataset (\Cref{subfig:3skinnonskin}) behaves differently: as $r$ increases, the matrix becomes ill-conditioned, the relative error drops to nearly the level of machine precision, and the plain Nystr\"om method fails beyond $r = 240$ as the Cholesky factorization breaks down. In this regime, the truncated implementation (\Cref{Alg:stableNystrom}) achieves an error that is roughly two order of magnitude smaller than the shifted variant, while the backslash implementation performs in between the two. To test the methods on more strongly low-rank kernels, we next increase the bandwidth to $\sigma = 30\sqrt{d}$, where $d$ denotes the feature dimension.

\paragraph{Large Bandwidth ($\sigma = 30\sqrt{d}$)} For the second experiment, we choose $\sigma = 30\sqrt{d}$, making the kernels closer to low-rank. The results are summarized in \Cref{fig:kerneltest2}. For \texttt{ijcnn1}, the four methods again produce essentially indistinguishable accuracy across all ranks. We observe the same behavior for all the other datasets not shown here. The remaining plots show that for \texttt{skin\_nonskin}, \texttt{cod-rna}, and \texttt{cadata}, the plain Nystr\"om method fails due to Cholesky breakdown and the stabilized algorithm consistently achieves errors that are one to two orders of magnitude smaller than the shifted variant for larger ranks. The backslash implementation shows more irregular behavior, sometimes performing better or worse than the shifted and truncated methods.

\begin{figure}[!ht]
\vspace{-0.5cm}
\subfloat[\centering \texttt{ijcnn1}]{\label{subfig:30ijcnn1}\includegraphics[scale = 0.5]{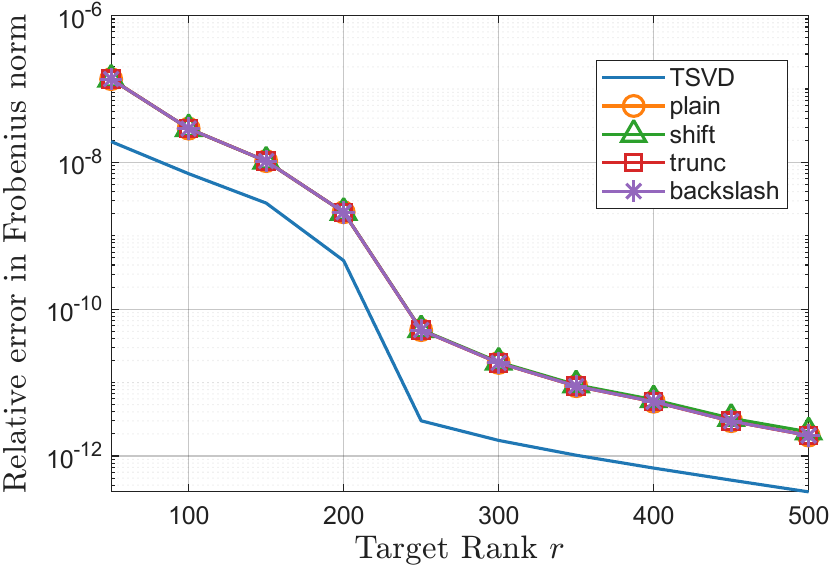}}\hfill
\subfloat[\centering \texttt{skin\_nonskin}]{\label{subfig:30skinnonskin}\includegraphics[scale = 0.5]{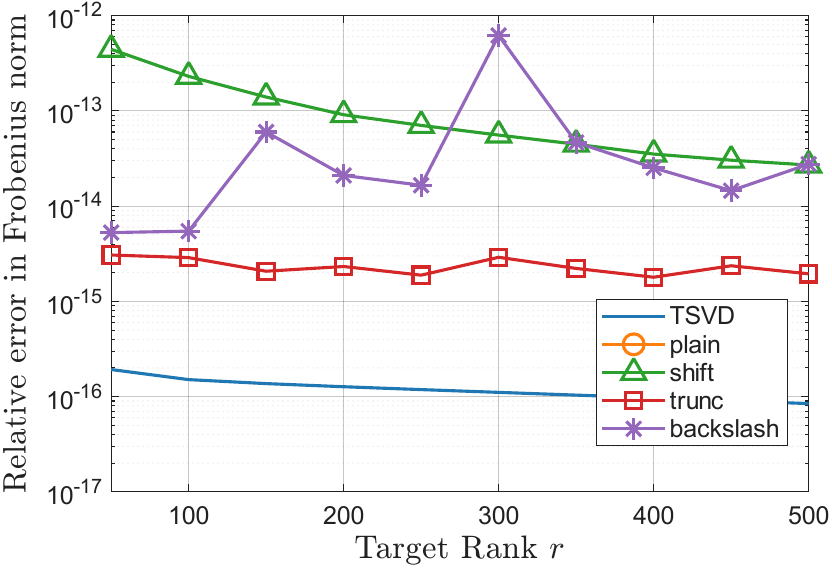}} \vspace{-0.8cm} \\ 
\subfloat[\centering \texttt{cod-rna}]{\label{subfig:30cod-rna}\includegraphics[scale = 0.5]{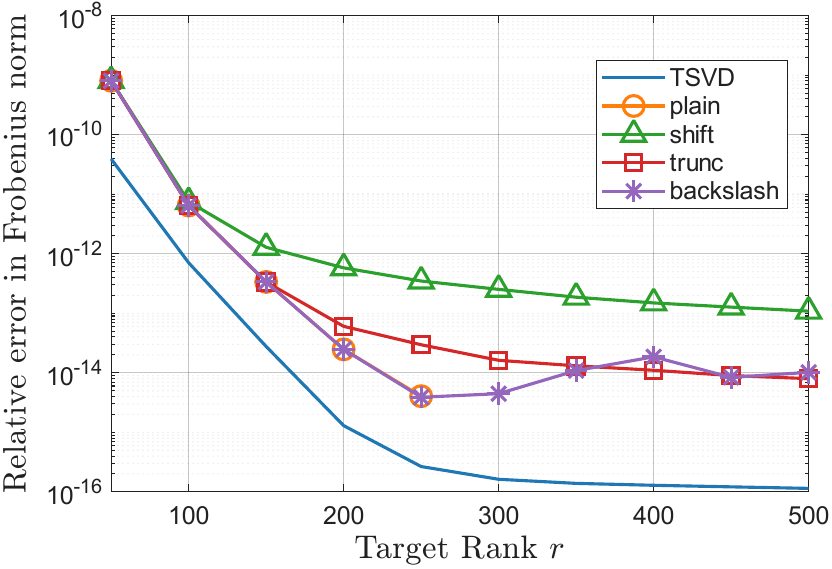}} \hfill
\subfloat[\centering \texttt{cadata}]{\label{subfig:30cadata}\includegraphics[scale = 0.5]{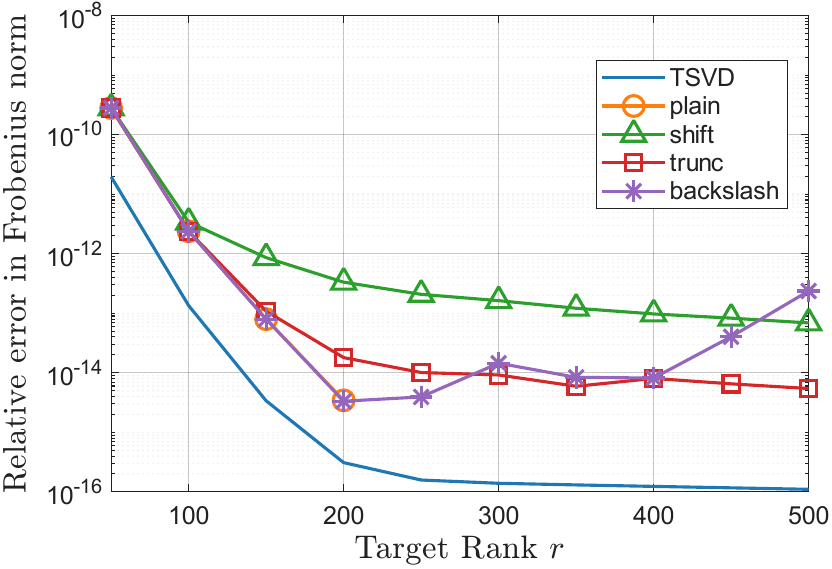}}
\centering 
\vspace{-0.2cm}
\caption{Relative approximation error versus rank $r$ for kernels with $\sigma = 30\sqrt{d}$.  Top left (\texttt{ijcnn1}): all four methods again coincide. The remaining panels show the behavior for \texttt{skin\_nonskin}, \texttt{cod-rna}, and \texttt{cadata}: the plain method fails at moderate ranks due to breakdown in the Cholesky factorization, while the truncated method consistently achieves lower error than the shifted variant for larger target rank. The behaviour for the backslash implementation is more irregular.}
\label{fig:kerneltest2}
\end{figure}

\paragraph{Uniform sampling (poor index selection)} In the final experiment, we keep $\sigma = 30\sqrt{d}$, but use a potentially poor set of indices obtained via uniform column sampling. In this setting, the plain Nystr\"om method often fails due to Cholesky breakdown. When this occurs, the backslash implementation either fails as well or becomes unstable, producing errors several orders of magnitude larger than those of the shifted and truncated implementations. The shifted and truncated implementations yield essentially the same accuracy, and the observed instability and error stagnation in these methods are likely due to the use of a poor set of indices.

\begin{figure}[!ht]
\vspace{-0.5cm}
\subfloat[\centering \texttt{a9a}]{\label{subfig:a9a}\includegraphics[scale = 0.5]{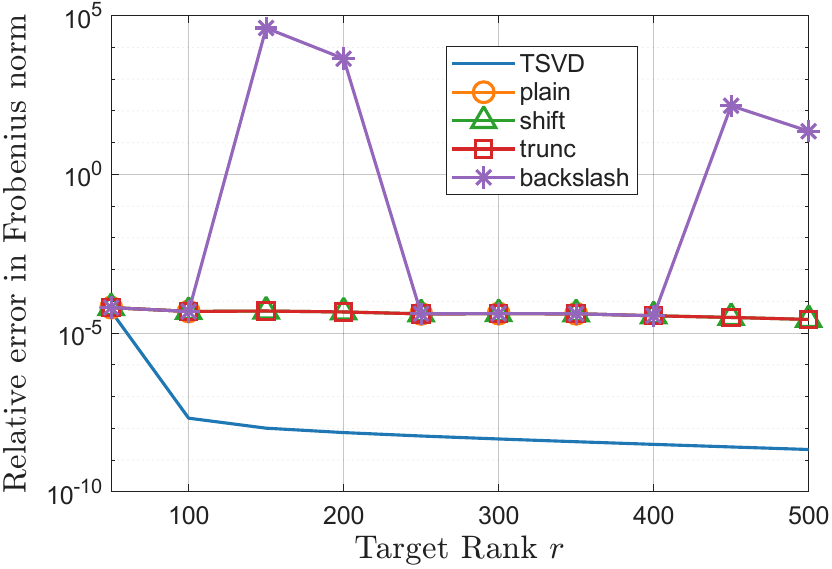}} \hfill
\subfloat[\centering \texttt{phishing}]{\label{subfig:phising}\includegraphics[scale = 0.5]{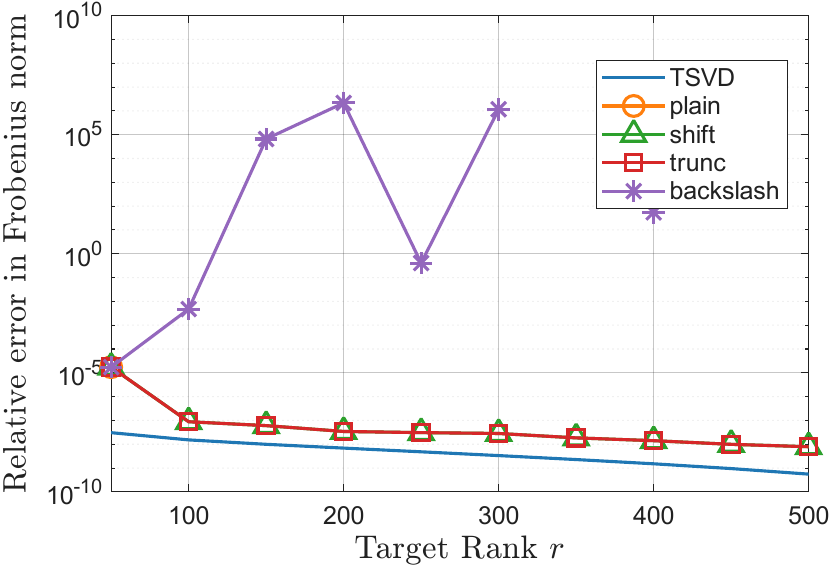}} 
\centering 
\vspace{-0.2cm}
\caption{Relative error in Frobenius norm versus target rank $r$ for kernels with $\sigma = 30\sqrt{d}$ using \emph{uniform} sampling. In both datasets, \texttt{a9a} and \texttt{phishing}, the plain and backslash implementations fail or become unstable, while the shift and truncated versions yield similar accuracy throughout.}
\label{fig:kerneltest3}
\end{figure}

Beyond the accuracy advantages shown, the stabilized Nystr\"om algorithm (\Cref{Alg:stableNystrom}) also offers lower computational cost than the shifted variant because it avoids the QR factorization of $AS$, as discussed in \Cref{subsec:contribution}. In double precision, the performance gap is visible only when the approximation error approaches machine precision; for moderate accuracy (i.e. for small $r$) all implementations give comparable results. However, for computations in lower precision, say single or half precision, the numerical accuracy is limited by a much larger machine precision, and the differences among the methods may become pronounced even when a relatively modest accuracy is acceptable.

\section*{Acknowledgments}
We thank Ethan Epperly for his helpful suggestions on the use of locally maximum-volume indices.

\bibliographystyle{siamplain}
\bibliography{references}

\end{document}